\documentclass[runningheads]{llncs}
\usepackage{makeidx}
\usepackage{amsmath,amssymb}
\usepackage{verbatim}

\usepackage{exscale}
\usepackage{tabularx}
\usepackage{latexsym}
\usepackage[usenames,dvipsnames]{color}
\usepackage{graphicx,epsfig}
\usepackage{hyperref}
\usepackage{url}
\usepackage{graphicx}
\newcommand{\mb}[1]{\mbox{\boldmath $#1$}}

\newcommand{\bea}{\begin{eqnarray}}
\newcommand{\eea}{\end{eqnarray}}

\newcommand{\vect}[1]{{\boldsymbol #1 }}

\newcommand{\inprod}[2]{\langle #1 , #2 \rangle }

\newcommand{\bc}{\begin{center}}
	\newcommand{\ec}{\end{center}}

\newcommand{\bY}{\vect{Y}}
\newcommand{\refs}[1]{$(\ref{#1})$}

\newcommand{\bX}{\vect{X}}

\newcommand{\bZ}{\boldsymbol Z}
\newcommand{\R}{\mathbb R}

\newcommand{\be}{\begin{equation}}
\newcommand{\ee}{\end{equation}}
\newcommand{\beaa}{\begin{eqnarray*}}
	\newcommand{\eeaa}{\end{eqnarray*}}
\newcommand{\ben}{\begin{enumerate}}
	\newcommand{\een}{\end{enumerate}}
\newcommand{\db}{\hspace*{\fill}{\zapf o}}

\newcommand{\bpn}{\begin{proposition}\twlsf}
	\newcommand{\epn}{\db\end{proposition}}
\newcommand{\bdm}{\begin{displaymath}}
\newcommand{\edm}{\end{displaymath}}
\newcommand{\ba}{\begin{array}}
	\newcommand{\ea}{\end{array}}

\newcommand{\st}{\mathop{\rm s.t.}}

\newcommand{\eps}{\epsilon}

\newcommand{\norm}[1]{\left\lVert#1\right\rVert}
\newcommand{\tnorm}[1]{\lVert\mkern-2mu |#1|\mkern-2mu\rVert}

\begin{document}
\title{Low-rank matrix recovery with Ky Fan $2$-$k$-norm\thanks{This work is partially supported by the Alan Turing Fellowship of the first author.}}
%
%
\author{Xuan Vinh Doan\inst{1,2}
\and  Stephen Vavasis\inst{3}
}
\authorrunning{X. V. Doan and S. Vavasis}
%
\institute{Operations Group, Warwick Business School, University of Warwick, Coventry, CV4 7AL, United Kingdom \\
\email{Xuan.Doan@wbs.ac.uk} \and
The Alan Turing Institute, British Library, 96 Euston Road, London NW1 2DB, United Kingdom \and 
Combinatorics and Optimization, University of Waterloo, 200 University Avenue West, Waterloo, ON N2L 3G1, Canada\\
\email{vavasis@uwaterloo.ca}}
\maketitle              
\begin{abstract}
We propose Ky Fan $2$-$k$-norm-based models for the non-convex low-rank matrix recovery problem. A general difference of convex algorithm (DCA) is developed to solve these models. Numerical results show that the proposed models achieve high recoverability rates.

\end{abstract}
\section{Introduction}
\label{sec:intro}
Matrix recovery problem concerns the construction of a matrix from incomplete information of its entries. This problem has a wide range of applications such as recommendation systems with incomplete information of users' ratings or sensor localization problem with partially observed distance matrices (see, e.g., \cite{candes2009exact}). In these applications, the matrix is usually known to be (approximately) low-rank. Finding these low-rank matrices are theoretically difficult due to their non-convex properties. Computationally, it is important to study the tractability of these problems given the large
scale of datasets considered in practical applications. Recht et al. \cite{RechtFazelParrilo} studied the low-rank matrix recovery problem using a convex relaxation approach which is tractable. More precisely, in order to recover a low-rank matrix $\mb{X}\in\mathbb{R}^{m\times n}$ which satisfy ${\cal A}(\mb{X})=\mb{b}$, where the linear map ${\cal A}:\R^{m\times n}\rightarrow\R^p$ and $\mb{b}\in\R^p$, $\mb{b}\neq\mb{0}$, are given, the following convex optimization problem is proposed:
\be
\label{eq:nuc}
\ba{rl}
\displaystyle\min_{\bX} & \norm{\bX}_*\\
\st & {\cal {A}}(\bX)=\mb{b},
\ea
\ee
where $\displaystyle\norm{\bX}_*=\sum_{i}\sigma_i(\bX)$ is the nuclear norm, the sum of all singular values of $\bX$. Recht et al. \cite{RechtFazelParrilo} showed the recoverability of this convex approach using some restricted isometry conditions of the linear map $\cal A$. In general, these restricted isometry conditions are not satisfied and the proposed convex relaxation can fail to recover the matrix $\bX$.

Low-rank matrices appear to be appropriate representations of data in other applications such
as biclustering of gene expression data. Doan and Vavasis \cite{doan2016finding} proposed a convex approach to recover low-rank clusters using dual Ky Fan $2$-$k$-norm instead of the nuclear norm. Ky Fan $2$-$k$-norm is defined as 
\be
\label{eq:kyfank2}
\tnorm{\mb{A}}_{k,2}=\left(\sum_{i=1}^k\sigma_i^2(\mb{A})\right)^{\frac{1}{2}},
\ee
where $\sigma_1\geq\ldots\sigma_k\geq 0$ are the first $k$ largest singular values of $\mb{A}$, $k\leq k_0=\mbox{rank}(\mb{A})$. The dual norm of the Ky Fan $2$-$k$-norm is denoted by $\tnorm{\,\cdot\,}_{k,2}^\star$,
\be
\label{eq:dkyfank2}
\ba{rl}
\displaystyle\tnorm{\mb{A}}_{k,2}^\star=\max_{\bX} & \inprod{\mb{A}}{\bX}\\
\st & \tnorm{\bX}_{k,2}\leq 1.
\ea
\ee
These unitarily invariant norms (see, e.g., Bhatia \cite{Bhatia97}) and their gauge functions have been used in sparse prediction problems \cite{Argyriou12}, low-rank regression analysis \cite{Giraud11} and multi-task learning regularization \cite{Jacob09}. When $k=1$, the Ky Fan $2$-$k$-norm is the spectral norm, $\norm{\mb{A}}=\sigma_1(\mb{A})$, the largest singular value of $\mb{A}$, whose dual norm is the nuclear norm. Similar to the nuclear norm, the dual Ky Fan $2$-$k$-norm with $k>1$ can be used to compute the $k$-approximation of a matrix $\mb{A}$ (Proposition 2.9, \cite{doan2016finding}), which demonstrates its low-rank property. Motivated by this low-rank property of the (dual) Ky Fan $2$-$k$-norm, which is more general than that of the nuclear norm, and its usage in other applications, in this paper, we propose a Ky Fan $2$-$k$-norm-based non-convex approach for the matrix recovery problem which aims to recover matrices which are not recoverable by the convex relaxation formulation \refs{eq:nuc}. In Section \ref{sec:models}, we discuss the proposed models in detail and in Section \ref{sec:comp}, we develop numerical algorithms to solve those models. Some numerical results will also be presented.   
\section{Ky Fan $2$-$k$-Norm-Based Models}
\label{sec:models}
The Ky Fan $2$-$k$-norm is the $\ell_2$-norm of the vector of $k$ largest singular values with $k\leq\min\{m,n\}$. Thus we have:
$$
\tnorm{\mb{A}}_{k,2}=\left(\sum_{i=1}^k\sigma_i^2(\mb{A})\right)^{\frac{1}{2}}\leq \norm{\mb{A}}_F=\left(\sum_{i=1}^{\min\{m,n\}}\sigma_i^2(\mb{A})\right)^{\frac{1}{2}},
$$
where $\norm{\cdot}$ is the Frobenius norm. Now consider the dual Ky Fan $2$-$k$-norm and use the definition of the dual norm, we obtain the following inequality:
$$
\norm{\mb{A}}_F^2=\inprod{\mb{A}}{\mb{A}}\leq\tnorm{\mb{A}}_{k,2}^{\star}\cdot\tnorm{\mb{A}}_{k,2}.
$$
Thus we have: 
\be
\label{eq:normin}
\tnorm{\mb{A}}_{k,2}\leq\norm{\mb{A}}_F\leq\tnorm{\mb{A}}_{k,2}^{\star}, \quad k\leq\min\{m,n\}.
\ee
It is clear that these inequalities become equalities if and only if $\text{rank}(\mb{A})\leq k$. It shows that to find a low-rank matrix $\bX$ that satisfies ${\cal A}(\bX)=\mb{b}$ with $\text{rank}(\bX)\leq k$, we can solve either the following optimization problem
\be
\label{eq:ratio}
\ba{rl}
\min & \displaystyle\frac{\tnorm{\bX}_{k,2}^{\star}}{\norm{\bX}_F}\\
\st & {\cal A}(\bX)=\mb{b},
\ea
\ee
or
\be
\label{eq:diff}
\ba{rl}
\min & \tnorm{\bX}_{k,2}^{\star}-\norm{\bX}_F\\
\st & {\cal A}(\bX)=\mb{b}.
\ea
\ee
It is straightforward to see that these non-convex optimization problems can be used to recover low-rank matrices as stated in the following theorem given the norm inequalities in \refs{eq:normin}.
\begin{theorem}
\label{thm:existence}
If there exists a matrix $\bX\in\R^{m\times n}$ such that $\mbox{rank}(\bX)\leq k$ and ${\cal A}(\bX)=\mb{b}$, then $\bX$ is an optimal solution of \refs{eq:ratio} and \refs{eq:diff}.
\end{theorem}
Given the result in Theorem \ref{thm:existence}, the exact recovery of a low-rank matrix using \refs{eq:ratio} or \refs{eq:diff} relies on the uniqueness of the low-rank solution of ${\cal A}(\bX)=\mb{b}$. Recht et al. \cite{RechtFazelParrilo} generalized the restricted isometry property of vectors introduced by Cand\`es and Tao \cite{candes2005decoding} to matrices and use it to provide sufficient conditions on the uniqueness of these solutions.

\begin{definition}[Recht et al. \cite{RechtFazelParrilo}]
For every integer $k$ with $1\leq k\leq \min\{m,n\}$, the $k$-restricted isometry constant is defined as the smallest number $\delta_k({\cal A})$ such that
\be
\label{eq:rip}
\left(1-\delta_k({\cal A})\right)\norm{\bX}_F\leq \norm{{\cal A}(\bX)}_2\leq\left(1+\delta_k({\cal A})\right)\norm{\bX}_F
\ee
holds for all matrices $\bX$ of rank at most $k$.
\end{definition}
Using Theorem 3.2 in Recht et al. \cite{RechtFazelParrilo}, we can obtain the following exact recovery result for \refs{eq:ratio} and \refs{eq:diff}.
\begin{theorem}
\label{thm:unique}
Suppose that $\delta_{2k}<1$ and there exists a matrix $\bX\in\mathbb{R}^{m\times n}$ which satisfies ${\cal A}(\bX)=\mb{b}$ and $\text{rank}(\bX)\leq k$, then $\bX$ is the unique solution to \refs{eq:ratio} and \refs{eq:diff}, which implies exact recoverability.
\end{theorem}

The condition in Theorem \ref{thm:unique} is indeed better than those obtained for the nuclear norm approach (see, e.g., Theorem 3.3 in Recht et al. \cite{RechtFazelParrilo}). The non-convex optimization problems \refs{eq:ratio} and \refs{eq:diff} use norm ratio and difference. When $k=1$, the norm ratio and difference are computed between the nuclear and Frobenius norm. The idea of using these norm ratios and differences with $k=1$ has been used to generate non-convex sparse generalizer in the vector case, i.e., $m=1$. Yin et al. \cite{yin2014ratio} investigated the ratio $\ell_1/\ell_2$ while Yin et al. \cite{yin2015minimization} analyzed the difference $\ell_1-\ell_2$ in compressed sensing. Note that even though optimization formulations based on these norm ratios and differences are non-convex, they are still relaxations of $\ell_0$-norm minimization problem unless the sparsity level of the optimal solution is $s=1$. Our proposed approach is similar to the idea of the truncated difference of the nuclear norm and Frobenius norm discussed in Ma et al \cite{ma2017truncated}. Given a parameter $t\geq 0$, the truncated difference is defined as $\displaystyle\norm{\mb{A}}_{*,t-F}=\sum_{i=t+1}^{\min\{m,n\}}\sigma_i(\mb{A})-\left(\sum_{i=t+1}^{\min\{m,n\}}\sigma_i^2(\mb{A})\right)^{\frac{1}{2}}\geq 0$. For $t\geq k-1$, the problem of truncated difference minimization can be used to recover matrices with rank at most $k$ given that $\norm{\bX}_{*,t-F}=0$ if $\text{rank}(\bX)\leq t+1$. Similar results for exact recovery as in Theorem \ref{thm:unique} are provided in Theorem 3.7(a) in Ma et al \cite{ma2017truncated}. Despite the similarity with respect to the recovery results, the problems \refs{eq:ratio} and \refs{eq:diff} are motivated from a different perspective. We are now going to discuss how to solve these problems next.
\section{Numerical Algorithm}
\label{sec:comp}
\subsection{Difference of Convex Algorithms}
\label{ssec:dca}
We start with the problem \refs{eq:ratio}. It can be reformulated as
\be
\label{eq:zprob}
\ba{rl}
\displaystyle\max_{\vect{Z},z} & \norm{\bZ}_F^2\\
\st & \tnorm{\bZ}_{k,2}^{\star}\leq 1,\\
& {\cal A}(\bZ)-z\cdot\mb{b}=\mb{0},\\
& z > 0,
\ea
\ee
with the change of variables $z=1/\tnorm{\bX}_{k,2}^{\star}$ and $\bZ = \bX/\tnorm{\bX}_{k,2}^{\star}$. The compact formulation 
\be
\label{eq:czprob}
\min_{\vect{Z},z}\,\delta_{{\cal Z}}(\bZ,z) - \norm{\bZ}_F^2/2,
\ee
where $\cal Z$ is the feasible set of the problem \refs{eq:zprob} and $\delta_{{\cal Z}}(\cdot)$ is the indicator function of $\cal Z$. The problem \refs{eq:czprob} is a difference of convex (d.c.) optimization problem (see, e.g. \cite{tao1997convex}). The differnce of convex algorithm DCA proposed in \cite{tao1997convex} can be applied to the problem \refs{eq:czprob} as follows.

\begin{quote}
\textbf{Step 1.} Start with $(\bZ^0,z^0)=(\bX^0/\tnorm{\bX^0}_{k,2}^{\star},1/\tnorm{\bX^0}_{k,2}^{\star})$ for some $\bX^0$ such that ${\cal A}(\bX^0)=\mb{b}$ and set $s=0$.

\noindent
\textbf{Step 2.} Update $(\bZ^{s+1},z^{s+1})$ as an optimal solution of the following convex optimization problem
\be
\label{eq:rsprob}
\ba{rl}
\displaystyle\max_{\vect{Z},z} & \inprod{\bZ^s}{\bZ} \\
\st & \tnorm{\bZ}_{k,2}^{\star}\leq 1\\
& {\cal A}(\bZ)-z\cdot \mb{b}=\mb{0}\\
& z>0.
\ea
\ee
\noindent
\textbf{Step 3.} Set $s\leftarrow s+1$ and repeat Step 2.
\end{quote}

Let $\bX^s=\bZ^s/z^s$ and use the general convergence analysis of DCA (see, e.g., Theorem 3.7 in \cite{tao1998dc}), we can obtain the following convergence results.
\begin{proposition}
\label{prop:conv}
Given the sequence $\{\bX^s\}$ obtained from the DCA algorithm for the problem \refs{eq:czprob}, the following statements are true.
\begin{itemize}
\item[(i)] The sequence $\displaystyle\left\{\frac{\tnorm{\bX^s}_{k,2}^{\star}}{\norm{\bX^s}_F}\right\} $ is non-increasing and convergent.
\item[(ii)] $\displaystyle\norm{\frac{\bX^{s+1}}{\tnorm{\bX^{s+1}}_{k,2}^{\star}}-\frac{\bX^{s}}{\tnorm{\bX^{s}}_{k,2}^{\star}}}_F\rightarrow 0 $ when $s\rightarrow\infty$.
\end{itemize}
\end{proposition}
The convergence results show that the DCA algorithm improves the objective of the ratio minimization problem \refs{eq:ratio}. The DCA algorithm can stop if $(\bZ^s,z^s)\in{\cal O}(\bZ^s)$, where ${\cal O}(\bZ^s)$ is the set of optimal solution of \ref{eq:rsprob} and $(\bZ^s,z^s)$ which satisfied this condition is called a critical point. Note that (local) optimal solutions of \refs{eq:czprob}can be shown to be critical points. The following proposition shows that an equivalent condition for critical points.

\begin{proposition}
\label{prop:ocond}
$(\bZ^s,z^s)\in{\cal O}(\bZ^s)$ if and only if $\bY=\mb{0}$ is an optimal solution of the following optimization problem
\be
\label{eq:null}
\ba{rl}
\displaystyle\min_{\bY} & \displaystyle\tnorm{\bX^s+\bY}_{k,2}^\star - \frac{\tnorm{\bX^s}_{k,2}^\star}{\norm{\bX^s}_F^2}\cdot\inprod{\bX^s}{\bY}\\
\st & {\cal A}(\bY)=\mb{0}.
\ea
\ee
\end{proposition}

\begin{proof}
Consider $\bY\in\mbox{Null}({\cal A})$, i.e., ${\cal A}(\bY)=\mb{0}$, we then have: $$\displaystyle\left(\frac{\bX^s+\bY}{\tnorm{\bX^s+\bY}_{k,2}^\star},\frac{1}{\tnorm{\bX^s+\bY}_{k,2}^\star}\right)\in{\cal Z}.$$
Clearly,
$\displaystyle\inprod{\frac{\bX^s}{\tnorm{\bX^s}_{k,2}^\star}}{\frac{\bX^s+\bY}{\tnorm{\bX^s+\bY}_{k,2}^\star}}\leq \inprod{\frac{\bX^s}{\tnorm{\bX^s}_{k,2}^\star}}{\frac{\bX^s}{\tnorm{\bX^s}_{k,2}^\star}}$  is equivalent to $$\displaystyle\tnorm{\bX^s+\bY}_{k,2}^\star-\frac{\tnorm{\bX^s}_{k,2}^\star}{\norm{\bX^s}_F^2}\cdot\inprod{\bX^s}{\bY}\geq\tnorm{\bX^s}_{k,2}^\star.
$$
When $\bY=\mb{0}$, we achieve the equality. We have: $(\bZ^s,z^s)\in{\cal O}(\bZ^s)$ if and only the above inequality holds for all $\bY \in\mbox{Null}({\cal A})$, which means $f(\bY;\bX^s)\geq f(\mb{0};\bX^s)$ for all $\bY \in\mbox{Null}({\cal A})$, where $f(\bY;\bX)=\displaystyle\tnorm{\bX+\bY}_{k,2}^\star-\frac{\tnorm{\bX}_{k,2}^\star}{\norm{\bX}_F^2}\cdot\inprod{\bX}{\bY}$. Clearly, it is equivalent to the fact that $\bY=\mb{0}$ is an optimal solution of \refs{eq:null}.
\end{proof}

The result of Proposition \ref{prop:ocond} shows the similarity between the norm ratio minimization problem \refs{eq:ratio} and the norm different minimization problem \refs{eq:diff} with respect to the implementation of the DCA algorithm. It is indeed that the problem \refs{eq:diff} is a d.c. optimization problem and the DCA algorithm can be applied as follows.
\begin{quote}
\textbf{Step 1.} Start with some $\bX^0$ such that ${\cal A}(\bX^0)=\mb{b}$ and set $s=0$.

\noindent
\textbf{Step 2.} Update $\bX^{s+1}=\bX^s+\bY$, where $\bY$ is an optimal solution of the following convex optimization problem
\be
\label{eq:dsprob}
\ba{rl}
\ba{rl}
\displaystyle\min_{\bY} & \displaystyle\tnorm{\bX^s+\bY}_{k,2}^\star - \frac{1}{\norm{\bX^s}_F}\cdot\inprod{\bX^s}{\bY}\\
\st & {\cal A}(\bY)=\mb{0}.
\ea
\ea
\ee
\noindent
\textbf{Step 3.} Set $s\leftarrow s+1$ and repeat Step 2.
\end{quote}
It is clear that $\bX^s$ is a critical point for the problem \refs{eq:diff} if and only if $\bY$ is an optimal solution of \refs{eq:dsprob}. Both problems \refs{eq:rsprob} and \refs{eq:dsprob} can be written in the general form as
\be
\label{eq:asprob}
\ba{rl}
\displaystyle\min_{\bY} & \displaystyle\tnorm{\bX^s+\bY}_{k,2}^\star - \alpha(\bX^s)\cdot\inprod{\bX^s}{\bY}\\
\st & {\cal A}(\bY)=\mb{0},
\ea
\ee
where $\displaystyle\alpha(\bX)=\frac{\tnorm{\bX^s}_{k,2}^\star}{\norm{\bX^s}_F^2}$ for \refs{eq:rsprob} and $\displaystyle\alpha(\bX)=\frac{1}{\norm{\bX^s}_F}$ for \refs{eq:dsprob}, respectively. Given that  ${\cal A}(\bX^s)=\mb{b}$, this problem can be written as 
\be
\label{eq:xsprob}
\ba{rl}
\displaystyle\min_{\bX} & \displaystyle\tnorm{\bX}_{k,2}^\star - \alpha(\bX^s)\cdot\inprod{\bX^s}{\bX}\\
\st & {\cal A}(\bX)=\mb{b}.
\ea
\ee 
The following proposition shows that $\bX^s$ is a critical point of the problem \refs{eq:xsprob} for many functions $\alpha(\cdot)$ if $\text{rank}(\bX^s)\leq k$.
\begin{proposition}
\label{prop:alpha}
If $\text{rank}(\bX^s)\leq k$, $\bX^s$ is a critical point of the problem \refs{eq:xsprob} for any function $\alpha(\cdot)$ which satisfies
\be
\frac{1}{\norm{\bX}_F}\leq\alpha(\bX)\leq\frac{\tnorm{\bX^s}_{k,2}^\star}{\norm{\bX^s}_F^2}.
\ee
\end{proposition}

\begin{proof}
If $\text{rank}(\bX^s)\leq k$, we have: $\alpha(\bX^s)=1/\tnorm{\bX^s}_{k,2}$ since $\tnorm{\bX^s}_{k,2}=\norm{\bX}_F=\tnorm{\bX^s}_{k,2}^{\star}$. Given that 
$$
\partial\tnorm{\mb{A}}_{k,2}^{\star}=\arg\max_{\bX:\tnorm{\bX}_{k,2}\leq 1}\inprod{\bX}{\mb{A}},
$$ 
we have: $\alpha(\bX^s)\cdot\bX^s\in\partial\tnorm{\bX^s}_{k,2}^{\star}$. Thus for all $\bY$, the following inequality holds:
$$
\tnorm{\bX^s+\bY}_{k,2}^{\star} - \tnorm{\bX^s}_{k,2}^{\star}\geq \inprod{\alpha(\bX^s)\cdot\bX^s}{\bY}.
$$
It implies $\bY=\mb{0}$ is an optimal solution of the problem \refs{eq:asprob} since the optimality condition is
$$
\tnorm{\bX^s+\bY}_{k,2}^{\star} - \tnorm{\bX^s}_{k,2}^{\star}\geq \inprod{\alpha(\bX^s)\cdot\bX^s}{\bY},\quad\forall\,\bY\,:\,{\cal A}(\bY) = \mb{0}.
$$
Thus $\bX^s$ is a critical point of the problem \refs{eq:xsprob}.
\end{proof}
Proposition \ref{prop:alpha} shows that one can choose different functions $\alpha(\cdot)$ such as $\alpha(\bX)=1/\tnorm{\bX}_{k,2}$ for the sub-problem in the general DCA framework to solve the original problem. This generalized sub-problem \refs{eq:xsprob} is a convex optimization problem, which can be formulated as a semidefinite optimization problem given the following calculation of the dual Ky Fan $2$-$k$-norm provided in \cite{doan2016finding}:
\be
\label{eq:msk2dual}
\ba{rl}
\tnorm{\mb{X}}_{k,2}^\star = \min & p + \mbox{trace}(\mb{R})\\
\st & kp - \mbox{trace}(\mb{P}) = 0,\\
\quad & p\mb{I} -\mb{P}\succeq 0,\\
\quad & \begin{pmatrix}
	\mb{P} & -\frac{1}{2}\mb{X}^T\\
	-\frac{1}{2}\mb{X} & \mb{R}
\end{pmatrix}\succeq 0.
\ea
\ee
In order to implement the DCA algorithm, one also needs to consider how to find the initial solution $\bX^0$. We can use the nuclear norm minimization problem \ref{eq:nuc}, the convex relaxation of the rank minimization problem, to find $\bX^0$. A similar approach is to use the following dual Ky Fan $2$-$k$-norm minimization problem to find $\bX^0$ given its low-rank properties:   
\be
\label{eq:kyfan}
\ba{rl}
\displaystyle\min_{\bX} & \tnorm{\bX}_{k,2}^{\star}\\
\st & {\cal {A}}(\bX)=\mb{b}.
\ea
\ee 
This initial problem can be considered as an instance of \refs{eq:xsprob} with $\bX^s=\mb{0}$ (and $\alpha(\mb{0})=1$), which is equivalent to starting the iterative algorithm with $\bX^0=\mb{0}$ one step ahead. We are now ready to provide some numerical results.
\subsection{Numerical Results}
\label{ssec:num}
Similar to Cand\`es and Recht \cite{candes2009exact}, we construct the following the experiment. We generate $\mb{M}$, an $m\times n$ matrix of rank $r$, by sampling two $m\times r$ and $n\times r$ factors $\mb{M}_L$ and $\mb{M}_R$ with i.i.d. Gaussian entries and setting $\mb{M} = \mb{M}_L\mb{M}_R$. The linear map $\cal A$ is constructed with $s$ independent Gaussian matrices $\mb{A}_i$ whose entries follows ${\cal N}(0,1/s)$, i.e.,
$$
{\cal A}(\bX)=\mb{b}\,\Leftrightarrow\,\inprod{\mb{A}_i}{\bX}=\inprod{\mb{A}_i}{\mb{M}} = b_i,\quad i=1,\ldots,s.
$$
We generate $K=50$ matrix $\mb{M}$ with $m=50$, $n=40$, and $r=2$. The dimension of these matrices is $d_r=r(m+n-r)=176$. For each $\mb{M}$, we generate $s$ matrices for the random linear map with $s$ ranging from $180$ to $500$. We set the maximum number of iterations of the algorithm to be $N_{\max} = 100$. The instances are solved using SDPT3 solver \cite{toh1999sdpt3} for semi-definite optimization problems in Matlab. The computer used for these numerical experiments is a 64-bit Windows 10 machine with 3.70GHz quad-core CPU, and 32GB RAM. The performance measure is the relative error $\displaystyle\frac{\norm{\bX-\mb{M}}_F}{\norm{\mb{M}}_F}$ and the threshold $\eps=10^{-6}$ is chosen. We run three different algorithms, \texttt{nuclear} used the nuclear optimization formulation \refs{eq:nuc}, \texttt{k2-nuclear} used the proposed iterative algorithm with initial solution obtained from \refs{eq:nuc}, and \texttt{k2-zero} used the same algorithm with initial solution $\bX^0=\mb{0}$. Figure \ref{fig:ptr2} shows recovery probabilities and average computation times (in seconds) for different sizes of the linear map.

\begin{figure}[ht]
	\includegraphics[width=\textwidth]{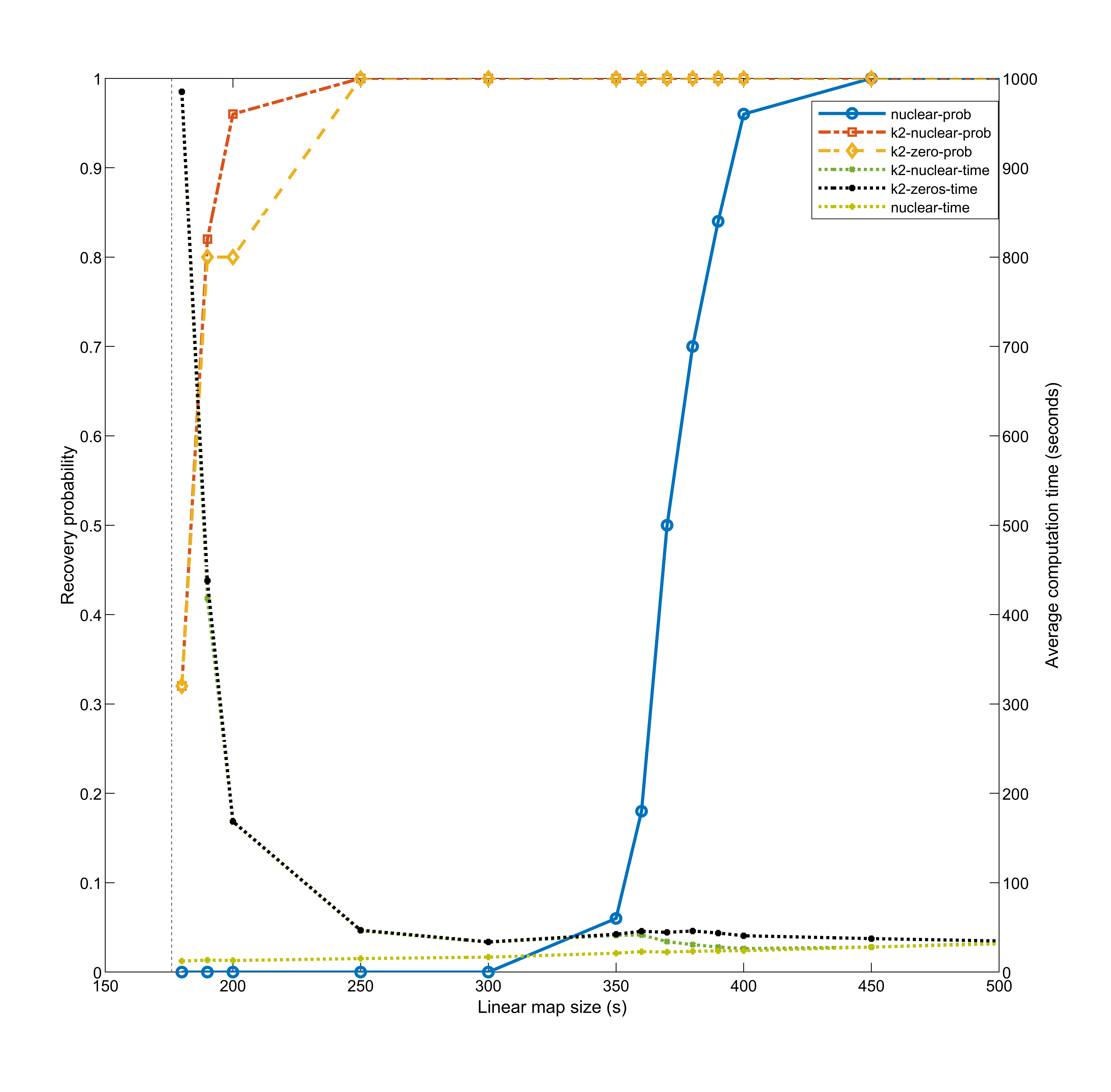}
	\caption{Recovery probabilities and average computation times of three algorithms } 
	\label{fig:ptr2}
\end{figure}

The results show that the proposed algorithm can recover exactly the matrix $\mb{M}$ with $100\%$ rate when $s\geq 250$ with both initial solutions while the nuclear norm approach cannot recover any matrix at all, i.e., $0\%$ rate, if $s\leq 300$. \texttt{k2-nuclear} is slightly better than \texttt{k2-zero} in terms of recoverability when $s$ is small while their average computational times are almost the same in all cases. The efficiency of the proposed algorithm when $s$ is small comes with higher average computational times as compared to that of the nuclear norm approach. For example, when $s=180$, on average, one needs $80$ iterations to reach the solution when the proposed algorithm is used instead of $1$ with the nuclear norm optimization approach. Note that the average number of iterations is computed for all cases including cases when the matix $\mb{M}$ cannot be recovered. For recoverable cases, the average number of iterations is much less. For example, when $s=180$, the average number of iterations for recoverable case is $40$ instead of $80$. When the size of the linear map increases, the average number of iterations is decreased significantly. We only need $2$ extra iterations when $s=250$ or $1$ extra iteration on average when $s=300$ to obtain $100\%$ recover rate when the nuclear norm optimization approach still cannot recover any of the matrices ($0\%$ rate). These results show that the proposed algorithm achieve significantly better recovery rate with a small number of extra iterations in many cases. We also test the algorithms with higher ranks including $r=5$ and $r=10$. Figure \ref{fig:ptmr} shows the results when the size of linear map is $s=\lceil 1.05d_r\rceil$.

\begin{figure}[ht]
	\includegraphics[width=\textwidth]{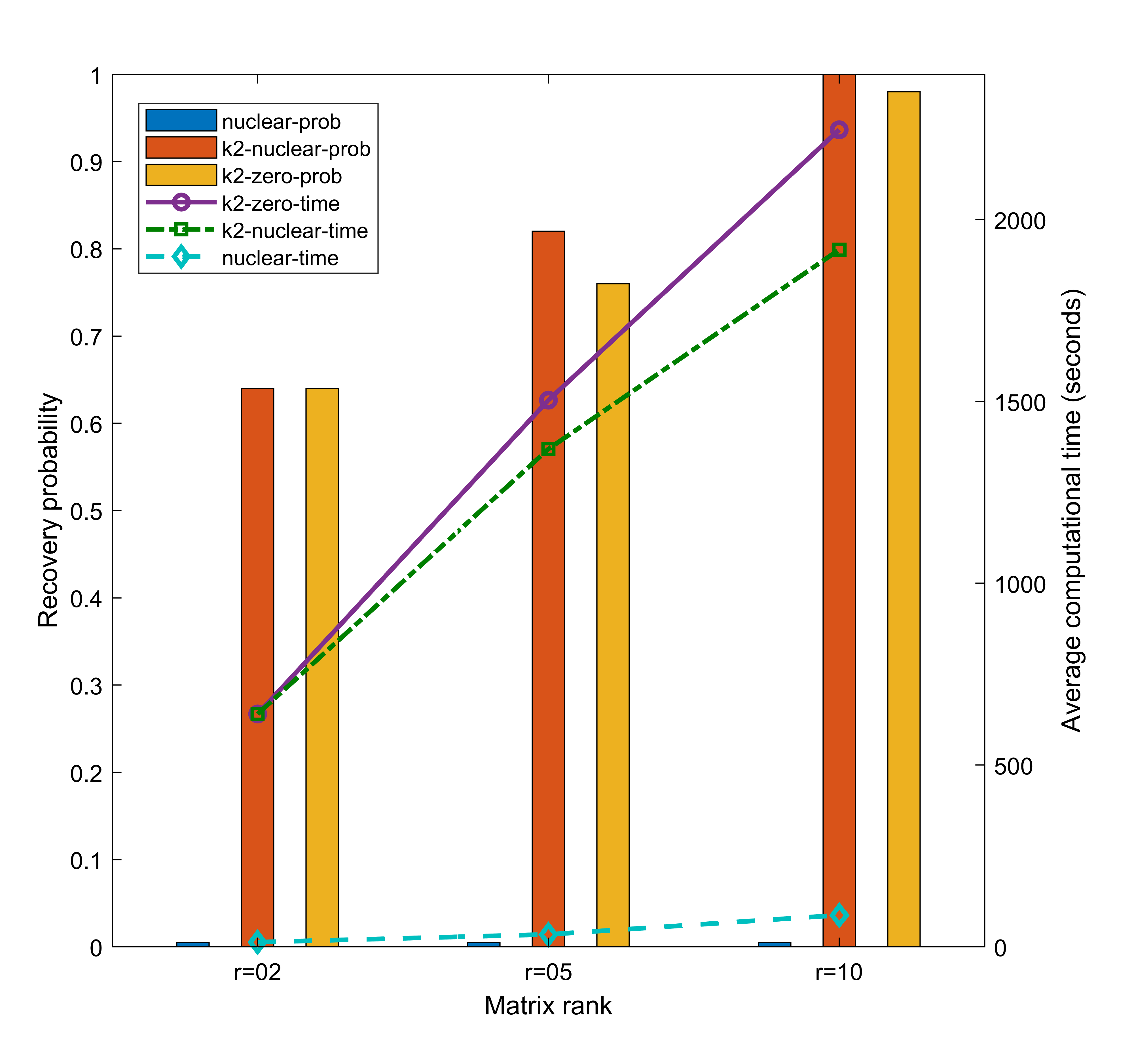}
	\caption{Recovery probabilities and average computation times for different ranks } 
	\label{fig:ptmr}
\end{figure}

These results show that when the size of linear maps is small, the proposed algorithms are significantly better than the nuclear norm optimization approach. With $s=\lceil 1.05d_r\rceil$, the recovery probability increases when $r$ increases and it is close to 1 when $r=10$. The computational time increases when $r$ increases given that the size of the sub-problems depends on the size of the linear map. With respect to the number of iterations, it remains low. When $r=10$, the average numbers of iterations are 22 and 26 for \texttt{k2-nuclear} and \texttt{k2-zero}, respectively. It shows that \texttt{k2-nuclear} is slightly better than \texttt{k2-zero} both in terms of recovery probability and computational time.
\section{Conclusion}
\label{sec:con}
We have proposed non-convex models based on the dual Ky Fan $2$-$k$-norm for low-rank matrix recovery and developed a general DCA framework to solve the models. The computational results are promising.  Numerical experiments with larger instances will be  conducted with first-order algorithm development for the proposed modes as a future research direction.
%
%
%
\bibliographystyle{splncs04}
\bibliography{WCGO}
%
%
%
%
%
\end{document}